\documentclass[12pt]{article}
\usepackage{amsmath,amsfonts,amssymb,amsthm}
\usepackage{graphicx}
\usepackage{url} 
\usepackage{setspace}
\newtheorem{proposition}{Proposition}[section]
\newtheorem{theorem}[proposition]{Theorem}
\newtheorem{corollary}[proposition]{Corollary}
\newtheorem{lemma}[proposition]{Lemma}
\newtheorem{remark}[proposition]{Remark}

\newcommand{\nc}{\newcommand}
\nc{\I}{{\bf 1}}
\nc{\bG}{{G}}
\nc{\bS}{{\mathbf S}}
\nc{\bN}{{\mathbf N}}
\nc{\bM}{{\mathbf M}}
\nc{\cB}{{\mathcal B}}
\nc{\cG}{{\mathcal G}}
\nc{\cS}{{\mathcal S}}
\nc{\cM}{{\mathcal M}}
\nc{\R}{{\mathbb R}}
\nc{\N}{{\mathbb N}}
\nc{\Z}{{\mathbb Z}}

\DeclareMathOperator{\supp}{supp}

\nc{\BP}{\mathbb{P}}
\nc{\BE}{\mathbb{E}}
\nc{\BQ}{\mathbb{Q}}

\numberwithin{equation}{section}

\begin{document} 

\renewcommand{\thefootnote}{\fnsymbol{footnote}}
\author{{\sc G\"unter Last\thanks{Karlsruhe Institute of Technology,
    Germany. E-mail: \texttt{guenter.last@kit.edu}}, 
Wenpin Tang\thanks{University of Berkeley, 
    USA. E-mail: \texttt{wenpintang@stat.berkeley.edu}}
and Hermann Thorisson\thanks{University of Iceland, Iceland. 
E-mail: \texttt{hermann@hi.is}}}}
\title{Transporting random measures on the line\\
and embedding excursions into Brownian motion}
\date{\today}
\maketitle

\begin{abstract} 
\noindent 
We consider two jointly stationary and ergodic random
measures $\xi$ and $\eta$ on the real line $\R$ with equal intensities. 
An allocation is an equivariant random mapping from
$\R$ to $\R$. We give sufficient and partially necessary conditions
for the existence of allocations transporting $\xi$ to $\eta$.
An important ingredient of our approach is %to introduce
a transport kernel balancing $\xi$ and $\eta$, provided
these random measures are mutually singular.
In the second part of the paper, we apply this result to
the path decomposition of a two-sided Brownian motion
into three independent pieces: a time reversed Brownian motion on $(-\infty,0]$, 
an excursion distributed according to a conditional 
It\^{o} measure and a Brownian motion starting after this excursion.
An analogous result holds for Bismut's excursion measure. 
\end{abstract}

%\noindent
%{\bf R\'esum\'e.}
%On consid\`ere deux mesures al\'eatoires stationaires et ergodiques
%conjointement $\xi$ et $\eta$ sur le r\'eel $\mathbb{R}$ avec les
%intensit\'es \'egales.  Une allocation est une carte al\'eatoire
%\'equivariante de $\mathbb{R}$ \`a $\mathbb{R}$.  On donne des
%conditions suffisantes et n\'ecessaires partiellement pour l'existence
%des allocations transportant $\xi$ \`a $\eta$.  Un ingr\'edient
%important de notre approche est un noyau de transport balan\c cant
%$\xi$ et $\eta$, \'etant donn\'e que ces mesures al\'eatoires sont
%singuli\`eres mutuellement.  Dans la deuxi\`eme partie de cet article,
%on applique ce r\'esultat \`a la decomposition des trajectoires d'un
%mouvement brownien sym\'etrique en trois pi\`eces ind\'ependantes: un
%mouvement brownien renvers\'e dans le temps sur $(-\infty, 0]$, une
%excursion distribu\'ee selon une mesure conditionelle d'It\'o, et un
%mouvement brownien apr\`es cette excursion.  Un r\'esultat analogue
%est valable pour la measure d'excursion de Bismut.

\bigskip



\section{Introduction}

The  following extra head problem for a two-sided  sequence of  i.i.d.\ 
tosses of a fair coin was 
formulated and solved by Tom Liggett in the 2002 paper \cite{L02}:
can you shift 
the origin  to one of the heads 
in such a way that you have two independent one-sided  i.i.d.\ sequences,
one  to the left and one to the right of that head? 
Note that if you shift 
to the first head at or after the origin, 
then the sequence to the left of that head will be biased: the distance to the first head  to the left 
will {\em not} be geometric, it will be the sum of two independent geometric variables  minus~$1$ 
(this is the waiting time paradox).
Liggett's solution was both surprising and simple: 
If there is a head at the origin, do not shift. If there is a tail at the origin,
shift forward until you have equal number of heads and tails.
Then you are at a head and it is an extra head.

Here we shall consider the analogous problem of
finding extra excursions in a two-sided standard Brownian motion $B = (B_s)_{s\in\R}$.
Let $A$ be a measurable set of excursions (away from zero) having positive 
finite 
It\^{o} excursion measure. By an $A${\em -excursion} we mean an excursion 
 that is distributed according to 
the  It\^{o} excursion measure conditioned on $A$. 
An~{\em extra} $A$-excursion (starting at a random time $T$ and of length $X$) is an 
$A$-excursion $(B_{T+s})_{0\leq s < X}$
with the property that it is independent of $(B_{T-s})_{s\ge 0}$
and $(B_{T+X+s})_{s\ge 0}$ which are independent and both one-sided 
standard Brownian motions; we also call this {\em unbiased} embedding of the excursion. 

It is readily checked that  
there is a.s.\ a first excursion to the right of the origin with property
$A$ and that this excursion 
is an $A$-excursion. But it is not an {\em extra} $A$-excursion. 
Indeed, %we shall show that the Brownian motion $B$ has the following structure. It 
the Brownian motion $B$ splits a.s.\ into a two-sided
sequence of independent segments  such that: every odd-numbered
segment is an $A$-excursion;  every even-numbered segment except the one enumerated
$0$ is a standard Brownian motion starting from zero running until the first time that
an $A$-excursion occurs; but the segment enumerated $0$
consists of two independent segments of that type.  In addition to
this, the origin of $B$ is placed at
random in the segment enumerated $0$ according to the local time at
zero of the segment. More details on this picture are given in Sections 
\ref{firstexcursion} and \ref{extra}; see in particular 
Figure 2, Remark \ref{pointstat} and Remark \ref{rfirstexcursion}.

In order to find an extra $A$-excursion we need to extend 
the general  allocation (transport) theory for random measures that grew out of Liggett's 
original paper.
The shift %as 
described in the first paragraph,
when applied to all the tails, generates an {\em allocation} from 
tails  to heads; the allocation is {\em balancing} 
because it transports the counting measure for tails ({\em source}) 
into the counting measure for heads ({\em target}). 
In the recent papers \cite{LaMoeTho12,MortersRedl16} and \cite{PitmanTang2015}, 
balancing allocations for {\em diffuse} random measures on the line
were used for unbiased Skorohod embedding 
%in a two-sided Brownian motion 
%by a random time shift, 
and for unbiased 
embedding (by a random space-time shift) of the Brownian bridge. %in a two-sided Brownian motion 
In this paper we shall allow the target measure to be non-diffuse. 
This is needed because the target measure associated with the
$A$-excursions is a point process.

Before proceeding further, we need some notation.
Let $\xi$ and $\eta$ be two jointly stationary and ergodic random
measures on $\R$ with finite intensities 
$\lambda_\xi:=\BE\xi[0,1]$ and $\lambda_\eta:=\BE\eta[0,1]$.
An {\em allocation}
is a random (jointly measurable) mapping $t\mapsto \tau(t)$ from
$\R$ to $\R\cup\{\infty\}$ which is equivariant
under joint shifts of $t$ and the underlying randomness;
see \eqref{allocation} for an exact definition.
%Note that we allow for the value $\infty$.
An allocation is said to {\em balance} the {\em source} \,$\xi$\, and the\, {\em target} 
$\eta$ if $\BP(\xi(\{s\in\R:\tau(s)=\infty\})>0)=0$
and the image measure of $\xi$ under $\tau$  is $\eta$; that is, 
\begin{align}\label{bal123}
\int \I\{\tau(s)\in C\}\,\xi(ds)=\eta(C),\quad C\in\cB(\mathbb{R}),\, \BP\text{-a.e.}
\end{align} 
The balancing property \eqref{bal123} implies easily that
\begin{align}\label{i2}
\lambda_\xi= \lambda_\eta.
\end{align}
%In the case $\lambda_\xi>0$, the equality $\lambda_\xi= \lambda_\eta$ is equivalent
%to $\BP(\xi(\{s\in\R:\tau(s)=\infty\})>0)=0$. In this case we 
%call $\tau$ {\em proper} (with respect to the source $\xi$).
%It is only in the proper case that all the mass of $\xi$ is %being 
%transported
%to $\eta$. If $\lambda_\xi>\lambda_\eta$, then some of the mass
%of $\xi$ is lost.
%In the case $\lambda_\xi=\lambda_\eta$ 
The random variable $\tau(0)$ can be used to construct
a {\em shift-coupling} (see \cite{AldousThor,Thor96,Thor00}) 
of the Palm versions of $\xi$ and $\eta$; see \cite{Mecke2,HP05,LaTho09}.

In this paper, we prove that
if the source $\xi$ is diffuse,
and if the source and the target are mutually singular,
then the equality \eqref{i2} is not only necessary but also sufficient
for the existence of a balancing allocation. 
%(note that if both $\xi$ and $\eta$ are 
%diffuse then the $\leq$ sign in \eqref{matching} below can be replaced by $=$):

\begin{theorem}\label{tmain2} Assume that $\xi$ and $\eta$
are mutually singular jointly stationary and ergodic random measures on $\R$ such that
$\xi$ is diffuse and $\lambda_\xi=\lambda_\eta$. Then the allocation
$\tau$ defined~by 
\begin{align}\label{matching}
\tau(s):=\inf\{t> s \colon \xi[s,t]\le \eta[s,t]\},\quad s\in\R,
\end{align}
balances $\xi$ and $\eta$.
\end{theorem}

%\noindent
%Note that when both $\xi$ and $\eta$ are 
%diffuse then the $\leq$ sign in \eqref{matching} can be replaced by $=$\,.

In order to establish Theorem \ref{tmain2},
we prove an even more general result, Theorem \ref{tmain1},
which does not require $\xi$ to be diffuse;
we construct a balancing {\em transport kernel},
provided that $\lambda_\xi= \lambda_\eta$ and that $\xi$ and $\eta$
are mutually singular. %; see Theorem \ref{tmain1}. 
This %extends and 
relies heavily on Theorem 5.1 from \cite{LaMoeTho12}, 
a precursor of Theorem \ref{tmain2}
where both $\xi$ and $\eta$ are assumed to be diffuse.

Transports of  random measures and point processes have been studied
on more general phase spaces.
For further background we refer to %on transports of  random measures and point processes, see 
\cite{Thor96,L02,HL01,HP05,LaTho09,LaMoeTho12, Huesmann16}. 
The existence of an extra head was implicit in an abstract group         
result in \cite{Thor96}, but in that paper there was no hint at an explicit    
pathwise method of finding an extra head.  
In \cite{L02,HL01}, the sources are counting and Lebesgue measures 
and the targets are Bernoulli and Poisson processes. 
In~\cite{HP05}, the source is Lebesgue measure and the target is a 
simple point process, in particular a Poisson process. 
In~\cite{LaMoeTho12}, the source and target are both diffuse random measures on the line,
in particular local times of Brownian motion. 
In Theorem~\ref{tmain2} above, the source is  diffuse but the target is general, 
and according to Theorem \ref{tmain1} below (see Remark~\ref{U}),
a balancing allocation is obtained through external randomization
in the case where both source and target are general. 
The paper \cite{LaTho09} develops a general transport theory 
for random measures
(on Abelian groups) with focus on transport kernels rather than only allocations. 
The allocations studied in the present paper have a certain 
property of right-stability; see \cite[Section 7]{LaMoeTho12}.
The mass of the source prefers to be allocated as close as possible.
The paper \cite{Huesmann16} pursues a different approach, based on
the minimization of expected transport costs (defined in the Palm  sense).
It is shown that if the expected transport cost is finite and the source is absolutely
continuous, then there exists
a unique optimal allocation that can be
locally approximated with solutions to the classical Monge problem 
(see \cite{Villani08}).

The paper is organised as follows. Section~\ref{statm} gives preliminaries 
on random measures, transport kernels and allocations.
Section~\ref{skernel} provides the main transport result,
Theorem~\ref{tmain1}. %, presenting a transport kernel that 
%balances general mutually singular source and target %measures 
%provided they have the same finite intensity. 
%and Corollary~\ref{tmain2}, presenting a balancing allocation 
%in the case when the source is diffuse.
We then turn to the application to Brownian motion. 
Section~\ref{Palm} contains the key Palm and shift-coupling 
result for the embedding, Proposition~\ref{lstt}. 
Section~\ref{firstexcursion} is devoted to excursion theory  
and discusses the embedding problem.  
%Finally, 
Section~\ref{extra} applies Proposition~\ref{lstt}  
to unbiased embedding of conditional It\^{o} measures. 
We also apply this proposition to Bismut's excursion measure,
a close relative of It\^{o}'s measure.

\section{Preliminaries}\label{statm}

Let $(\Omega,\mathcal{F},\BP)$ be a $\sigma$-finite measure space
with associated integral operator $\BE$.
A {\em random measure} (resp.\ {\em point process}) $\xi$ on $\R$ (equipped
with its Borel $\sigma$-field $\cB(\R)$) is a kernel 
from $\Omega$ to $\R$ such that $\xi(\omega,C)<\infty$
(resp.\ $\xi(\omega,C)\in\N_0$)
for $\BP$-a.e.\ $\omega$ and all compact $C\subset\R$.
We assume that
$(\Omega,\mathcal{F})$ is equipped with a {\em measurable flow}
$\theta_s\colon\Omega \to \Omega$, $s\in \R$. This is a family
of %measurable 
mappings such that $(\omega,s)\mapsto \theta_s\omega$ 
is measurable, $\theta_0$ is the identity on $\Omega$ and
\begin{align}\label{flow}
  \theta_s \circ \theta_t =\theta_{s+t},\quad s,t\in \R,
\end{align}
where $\circ$ denotes composition.

A kernel $\xi$ from $\Omega$ to $\R$ is said to be {\em invariant}  (or {\em flow-adapted})  if
\begin{align}\label{adapt}
\xi(\theta_t\omega,C-t)=\xi(\omega,C),\quad C\in\mathcal{B}(\mathbb{R}),t\in\R,
\BP\text{-a.e.\ $\omega\in\Omega$}.
\end{align}
We assume that the measure $\BP$ is {\em stationary}; that is
$$
\BP\circ\theta_s=\BP,\quad s\in\R,
$$
where $\theta_s$ is interpreted as a mapping from $\mathcal{F}$ to $\mathcal{F}$
in the usual way:
$$
\theta_sA:=\{\theta_s\omega:\omega\in A\},\quad A\in\mathcal{F},\, s\in\R.
$$
The {\em invariant} $\sigma$-field $\mathcal{I}\subset \mathcal{F}$ is the
class of all sets $A\in\mathcal{F}$ satisfying $\theta_sA=A$ for all $s\in\R$.
We also assume that $\BP$ is {\em ergodic}; that is for any
$A\in\mathcal{I}$, we have either $\BP(A)=0$ or $\BP(A^c)=0$.

\begin{remark}\label{r1}\rm The assumption of ergodicity has been made
for simplicity and can be relaxed. The assumption $\lambda_\xi= \lambda_\eta$
has then to be replaced by
$$
\BE[\xi[0,1]\mid \mathcal{I}]=\BE[\eta[0,1]\mid \mathcal{I}],\quad \BP\text{-a.e.}
$$
We refer to \cite{LaTho09} for more detail on this point.
\end{remark}

%A measure $\mu$ on $\R$ is {\em diffuse} if $\mu\{t\}=0$ for all $t\in\R$
%and {\em purely discrete}, if there is a discrete (that is finite or countably infinite)
%set $D\subset\R$ such that $\mu(\R\setminus D)=0$. Any $\mu$ can
%be written in a unique way as the sum $\mu=\mu^c+\mu^d$ of a diffuse measure
%$\mu^c$ and a purely discrete measure $\mu^d$. If $\xi$ is an invariant random
%measure, then so are $\xi^c$ and $\xi^d$.

A {\em transport kernel} is a  sub-Markovian kernel
$K$ from $\Omega\times\R$ to $\R$.
A  transport kernel is {\em invariant} if 
\begin{align*}
K(\theta_s\omega,0,C-s)=K(\omega,s,C),\quad \quad  s\in\R,\,C\in\cB(\R),\, \text{$\BP$-a.e.\ $\omega\in\Omega$}.
\end{align*}
An {\em allocation} \cite{HP05,LaTho09} is a measurable mapping
$\tau\colon\Omega\times\R\rightarrow\R\cup\{\infty\}$ that is {\em equivariant}
in the sense that
\begin{align}\label{allocation}
\tau(\theta_t\omega,s-t)=\tau(\omega,s)-t,
\quad  s,t\in\R,\, \text{$\BP$-a.e.\ $\omega\in\Omega$}.
\end{align}
Any allocation
defines a transport kernel $K$ by $K(s,\cdot)=\I\{\tau(s)<\infty\}\delta_{\tau(s)}$. 

%\begin{remark}\label{r2}\rm In \cite{LaTho09} a transport
%kernel $K$ is Markovian, that is 
%$K(\omega,s,\R)=1$ for all $s\in\R$. We find it convenient to allow
%for $K(\omega,s,\R)<1$, that is for the loss of mass. In the same
%spirit we do not assume an allocation to take on only finite values,
%as it is the case in \cite{LaTho09,LaMoeTho12}.
%\end{remark}

\begin{remark}\label{r2}\rm In \cite{LaTho09}, a transport
kernel $K$ is Markovian; that is 
$K(\omega,s,\R)=1$ for all $s\in\R$. We find it convenient to allow
for $K(\omega,s,\R)<1$ on an exceptional set of points $(\omega,s)$. In the same
spirit, we do not assume an allocation to take on only finite values,
as it is the case in \cite{LaTho09,LaMoeTho12}.
\end{remark}

Let $\xi$ and $\eta$ be random measures on $\R$. We say that
a transport kernel $K$ {\em balances} $\xi$ and $\eta$ if 
$K(\omega,s,\R)=1$ a.e.\ w.r.t.\ the measure $\xi(\omega,ds)\BP(d\omega)$ and
$K$ transports $\xi$ to $\eta$, that is, 
\begin{align}\label{bal}
\int K(s,\cdot)\,\xi(ds)=\eta, \quad \BP\text{-a.e.}
\end{align}
If $\tau$ is an allocation such that the associated transport kernel $K$ 
%given by $K(\omega,s,\cdot)=\I\{\tau(\omega,s)\in\cdot\}$ 
balances $\xi$ and $\eta$, then we say that $\tau$ balances $\xi$ and $\eta$.

\section{Balancing mutually singular random measures}\label{skernel}

Throughout this section, let  $\xi$ and $\eta$  
be two invariant random measures defined on the $\sigma$-finite measure
space $(\Omega,\mathcal{F},\BP)$. We recall from the previous section that
$\BP$ is assumed to be stationary and ergodic under a given flow.
In particular, the joint distribution of $\xi$ and $\eta$ is stationary and ergodic.
We shall construct a transport kernel balancing $\xi$ and $\eta$. 
To this end, we use the following result 
from \cite[Theorem 5.1]{LaMoeTho12}
in a crucial way.

\begin{theorem}\label{LMT5.1} Assume that $\xi$ and $\eta$ are mutually
singular diffuse invariant random measures such that $\lambda_\xi= \lambda_\eta$.
Then the mapping $\tau$ defined by \eqref{matching}
is an allocation balancing $\xi$ and $\eta$. 
\end{theorem}

For any $u\in[0,1]$ we define a mapping $\tau^u\colon\Omega\times\R\rightarrow[0,\infty]$  by
\begin{align}\label{umatching}
\tau^u(s):=\inf\{t> s \colon u\xi\{s\}+\xi(s,t)\le \eta[s,t]\},\quad s\in\R,
\end{align}
where $\inf\emptyset:=\infty$. 

%If $\xi\{s\}=0$, then $\tau^u(s)=\tau(s)$, where
%$\tau(s)$ is given by \eqref{matching}.

\begin{theorem}\label{tmain1} Assume that $\xi$ and $\eta$ 
are mutually singular invariant random measures on $\R$ such that
$\lambda_\xi= \lambda_\eta$. Then
\begin{align}\label{kernelmass}
K(s,C):=\int^1_0\I\{\tau^u(s)\in C \}\,du,\quad s\in\R,\,C\in\cB(\R),
\end{align}
defines a transport kernel balancing $\xi$
and $\eta$. 
\end{theorem}

Since $\xi$ and $\eta$ are invariant, we obtain for all
$\omega$ outside a $\BP$-null set, for all
$s,t\in\R$ and for all $u\in[0,1]$ that
\begin{align*}
\tau^u(\theta_t\omega,s-t)&=\inf\{r> s-t \colon u\xi(\theta_t\omega,\{s-t\})
+\xi(\theta_t\omega,(s-t,r))\le \eta(\theta_t\omega,[s-t,r])\}\\
&=\inf\{r> s-t \colon u\xi(\omega,\{s\})
+\xi(\omega,(s,r+t))\le \eta(\omega,[s,r+t])\}\\
&=\tau^u(\omega,s)-t.
\end{align*}
Hence $\tau^u$ is an allocation and \eqref{kernelmass} defines a transport
kernel.

If $s\in\R$ satisfies $\xi\{s\}=0$, then $\tau^u(s)=\tau^0(s)$ does
not depend on $u\in[0,1]$. Therefore the kernel \eqref{kernelmass}
reduces on $\{s:\xi\{s\}=0\}$ to the allocation rule $\tau^0$, that is
\begin{align}\label{kernel2}
K(s,\cdot)=\I\{\xi\{s\}=0\}\delta_{\tau^0(s)}
+\I\{\xi\{s\}>0\}\int^1_0\I\{\tau^u(s)\in\cdot\}\,du.
\end{align}
If $\xi\{s\}>0$, then we may think of 
$u\xi\{s\}$ as a location picked at random in 
the mass of $\xi$ at $s$, before applying virtually the same rule 
$\tau$ as in Theorem \ref{LMT5.1}. If $\xi$ is diffuse, then
$\tau^0=\tau$, where $\tau$ is given by \eqref{matching}. Moreover,
the second term on the r.h.s.\ of \eqref{kernel2} vanishes in this case.
Thus, Theorem \ref{tmain1} implies Theorem \ref{tmain2}.

\begin{remark}\label{r8}\rm Theorem \ref{tmain2} is wrong without the assumption of 
mutual singularity. To see this, let $\xi'$ and $\eta'$ 
be mutually singular invariant random measures on $\R$ such that 
$\lambda_{\xi'}= \lambda_{\eta'}<\infty$. Assume that $\xi'$ is diffuse.
Let $\xi:=\xi'+\mu_1$
and $\eta:=\eta'+\mu_1$, where $\mu_1$ is Lebesgue measure on $\R$.
Then the allocation \eqref{matching} takes the form 
\begin{align*}
\tau(s):=\inf\{t> s \colon \xi'[s,t]\le \eta'[s,t]\},\quad s\in\R.
\end{align*}
By Theorem \ref{tmain2}, $\tau$ balances $\xi'$ and $\eta'$.
Therefore, $\tau$ balances $\xi$ and $\eta$ iff
\begin{align}\label{e3.49}
\int \I\{\tau(s)\in \cdot\}\,\mu_1(ds)=\mu_1,\quad \BP\text{-a.e.}
\end{align} 
This cannot be true in general. For a simple example let 
$\xi_0$ be Lebesgue measure on the set $A:=\cup_{i\in 3\Z}[i,i+2)$
and let $\eta_0$ be twice the Lebesgue measure on the set $\R\setminus A$.
Assume that $(\xi',\eta')=(\theta_U\xi_0,\theta_U\eta_0)$,
where $U$ is uniformly distributed on the interval $[0,3)$
and where we abuse notation by introducing
for any measure $\mu$ on $\R$ and $s\in\R$ a new measure
$\theta_s\mu$ by $\theta_s\mu:=\mu(\cdot +s)$.
For $s\in[U,U+2)$ we then have $\tau(s)=3U/2+3-s/2$, so that
$\int^{U+2}_U\I\{\tau(s)\in\cdot\}\,\mu_1(ds)$ is twice the
Lebesgue measure on $[U+2,U+3)$. Hence \eqref{e3.49} fails.
Note that \eqref{e3.49} fails,
even when modifying $\tau$ on the support of $\eta'$ in an arbitrary
manner.
\end{remark}

The proof of Theorem \ref{tmain1} relies 
on Theorem~\ref{LMT5.1} and the  six lemmas below. 
Of these lemmas all are deterministic except the final 
one, Lemma \ref{stat}. For convenience we assume that
$\xi$ and $\eta$ are locally finite everywhere on $\Omega$.
We shall use the decomposition $\xi=\xi^c+\xi^d$ of $\xi$
as the sum of its diffuse part $\xi^c$ and its purely discrete part $\xi^d$.
The formulas $\xi^c(dt):=\I\{\xi\{t\}=0\}\xi(dt)$ and
$\xi^d(dt):=\I\{\xi\{t\}>0\}\xi(dt)$ show that these
random measures are again invariant. Similar definitions
apply to $\eta$.

Theorem~\ref{LMT5.1} assumes $\xi$ and $\eta$ to be diffuse (and mutually singular). 
In order to obtain a diffuse source and target, 
we introduce a time change,
by stretching the real axis at the position of an atom by its size.
For that purpose we define, for $s\in\R$,
\begin{align*}
\zeta(s):=
\begin{cases}
s+\xi^d[0,s)+\eta^d[0,s),\quad &s\ge 0,\\
s-\xi^d[s,0)-\eta^d[s,0),\quad &s< 0.
%\zeta(t):=
%\begin{cases}
%t+\xi^d[0,t)+\eta^d[0,t),\quad &t\ge 0,\\
%t-\xi^d[t,0)-\eta^d[t,0),\quad &t< 0,
\end{cases}
\end{align*}
Define a random measure $\xi^*$ on $\R$ by
\begin{align*}
\xi^*(C):=\int\I\{\zeta(t)\in C\}\,\xi^c(dt)+
\iint \I\{\zeta(t)\le v\le \zeta(t)+\xi\{t\},v\in C\}
\xi\{t\}^{-1}\,dv\,\xi^d(dt).
\end{align*}  %% ** alternative formula
Define another random measure $\eta^*$ on $\R$
by replacing $\xi$ with $\eta$ in the above r.h.s.  
Then $\xi^*$ and $\eta^*$ are diffuse,
and it is easy to check that these random measures are again
mutually singular. 

To express $(\xi,\eta)$ in terms of $(\xi^*,\eta^*)$, we use
the generalized inverse $\zeta^{-1}$ of $\zeta$, defined~by
\begin{align*}
\zeta^{-1}(t):=\inf\{s\in\R\colon \zeta(s)\ge t\},\quad t\in\R.
\end{align*}
Since $\zeta$ is strictly increasing,  
%>H
the inverse time change
%<H
$\zeta^{-1}$ is continuous. 

%H>

\begin{lemma}\label{l1} Let $s\in\R$ and $v\in[0,\xi\{s\}\vee\eta\{s\}]$. Then
$\zeta^{-1}(\zeta(s)+v)=s$.
\end{lemma}
\begin{proof}
 Since $\zeta$ is (strictly) increasing, it 
is easy to prove the equivalence
\begin{align}\label{e3.91}
\zeta(s)\le t\; \Longleftrightarrow\; s\le \zeta^{-1}(t),
\end{align}
valid for all $s,t\in\R$. Applying this to the trivial
inequality $\zeta(s)\le \zeta(s)+v$ yields
$s\le \zeta^{-1}(\zeta(s)+v)$.
Assume by contradiction  that this inequality is strict, that is, 
$s< s'$ where $s' = \zeta^{-1}(\zeta(s)+v)$. 
Trivially, $s'\leq \zeta^{-1}(\zeta(s)+v)$ and  \eqref{e3.91} 
yields $\zeta(s')\le \zeta(s)+v$. 
This, together with $v \le \xi\{s\}\vee\eta\{s\}$, 
implies (for $s\ge 0$; the case $s<0$ is similar) that 
$$
s'+\xi^d[0,s')+\eta^d[0,s')\le s+\xi^d[0,s]+\eta^d[0,s]. 
$$
Now  $s'>s$ and $\xi^d[0,s')+\eta^d[0,s')\ge \xi^d[0,s]+\eta^d[0,s]$, which leads
to a contradiction. 
\end{proof}

\begin{lemma}\label{l2} Let $s_1<s_2$,  $v_1\in[0,\xi\{s_1\}\vee\eta\{s_1\}]$ 
and $v_2\in[0,\xi\{s_2\}\vee \eta\{s_2\}]$. Then 
$$
\xi^*[\zeta(s_1)+v_1,\zeta(s_2)+v_2]
=\I\{\eta\{s_1\}=0\}(\xi\{s_1\}-v_1)+\xi(s_1,s_2)+\I\{\eta\{s_2\}=0\}v_2.
$$
\end{lemma} %% ** add formula for \eta^*
\begin{proof}
 Since $\xi^*$ is diffuse, we have
\begin{align*}
\xi^*[\zeta(s_1)+v_1,\zeta(s_2)+v_2]
=\int \I\{\zeta(s_1)+v_1<t\le \zeta(s_2)+v_2\}\,\xi^*(dt)=I_1+I_2, 
\end{align*}
where
\begin{align*}
I_1&:=\int \I\{\zeta(s_1)+v_1<\zeta(s)\le \zeta(s_2)+v_2\}\,\xi^c(ds),\\
I_2&:=\sum_{s\in \R}\I\{\xi\{s\}>0\}
\int \I\{\zeta(s)< v\le \zeta(s)+\xi\{s\},
\zeta(s_1)+v_1<v\le \zeta(s_2)+v_2\}\,dv.
\end{align*}
%>H
Apply first  \eqref{e3.91} and then 
Lemma \ref{l1} to obtain %and \eqref{e3.91} we have 
%<H
\begin{align}\label{e3.17}
I_1=\int \I\{s_1<s\le s_2\}\,\xi^c(ds)=\xi^c(s_1,s_2).
\end{align}
Turning to $I_2$, we restrict ourselves to the
case $\eta\{s_1\}=\eta\{s_2\}=0$. The other cases
can be treated similarly. First note that the inequalities
$v\le \zeta(s)+\xi\{s\}$ and $\zeta(s_1)+v_1\le v$ imply 
$s_1\le s$ (by Lemma \ref{l1}), while the inequalities
$v\le \zeta(s_2)+v_2$ and $\zeta(s)\le v$ imply $s\le s_2$.
Splitting into the three cases\, $s = s_1$,\, $s_1 < s < s_2$,\, $s = s_2$\, yields
%Writing the interval $[s_1,s_2]$ as the union
%of $\{s_1\}$, $(s_1,s_2)$ and $\{s_2\}$ yields
\begin{align*}
I_2=&\I\{\xi\{s_1\}>0\}\int \I\{\zeta(s_1)+v_1<v\le \zeta(s_1)+\xi\{s_1\}\}\,dv\\
&+\sum_{s_1<s<s_2} \I\{\xi\{s\}>0\}
\int \I\{\zeta(s)< v\le \zeta(s)+\xi\{s\}\}\,dv\\
&+\I\{\xi\{s_2\}>0\}\int \I\{\zeta(s_2)<v\le \zeta(s_2)+v_2\}\,dv.
\end{align*}
It follows that
$$
I_2=\xi\{s_1\}-v_1+\xi^d(s_1,s_2)+v_2.
$$
Combining this with \eqref{e3.17} yields the assertion of
the lemma.
\end{proof}

According to the following change-of-variable result, $\zeta^{-1}$ balances $\xi^*$ and $\xi$.

\begin{lemma}\label{l3} Let $f\colon\R\to[0,\infty)$ be measurable. Then
\begin{align}\label{e3.4}
\int f(s)\,\xi(ds)=\int f(\zeta^{-1}(t))\,\xi^*(dt).
\end{align}
\end{lemma}
\begin{proof} %Let 
It suffices to establish \eqref{e3.4} for
$f:=\I_{[a,b)}$, where $a<b$.
Using \eqref{e3.91} we obtain
\begin{align*}
\int \I_{[a,b)}(\zeta^{-1}(t))\,\xi^*(dt)%&=\int\I\{a\le \zeta^{-1}(t)<b\}\,\xi^*(dt)\\
=\int\I\{\zeta(a)\le t<\zeta(b)\}\,\xi^*(dt)
=\xi[a,b),
\end{align*}
where we have used Lemma \ref{l2} (with $v_1=%\xi\{s_1\}$ and $
v_2=0$)
to get the second identity.
%This is enough to establish \eqref{e3.4} for general $f$.
\end{proof}

Define 
\begin{align}\label{matching*}
\tau^*(s):=\inf\{t> s \colon \xi^*[s,t]\le \eta^*[s,t]\},\quad s\in\R.
\end{align}

\begin{lemma}\label{l13} Let $s\in\R$ with 
$\xi\{s\}=\eta\{s\}=0$. Then
$\tau^*(\zeta(s))<\infty$ iff $\tau^0(s)<\infty$.
In this case
\begin{align}\label{e3.10}
\zeta^{-1}(\tau^*(\zeta(s)))=\tau^0(s).
\end{align}
\end{lemma}

\begin{proof}
%Take $s\in\R$ satisfying $\xi\{s\}=0$.
%In view of the assertion \eqref{e3.10}
%and since $\xi$ and $\eta$ are mutually singular, we
%can assume that $\eta\{s\}=0$.
We abbreviate $t^*:=\tau^*(\zeta(s))$. 

First consider the case $t^*=\zeta(s)$. Then $\zeta^{-1}(t^*)=s$
(by Lemma \ref{l1}) and we need to show that $\tau^0(s)=s$.
There are $t_n>t^*$, $n\in\N$, such that
$\xi^*[\zeta(s),t_n]\le \eta^*[\zeta(s),t_n]$ and $t_n\downarrow t^*$. 
We distinguish two cases.
In the first case, there are infinitely many $n\in\N$ such
that $t_n=\zeta(s_n)+v_n$ for some $s_n\in\R$ satisfying
$\eta\{s_n\}=0$ and $v_n\in[0,\xi\{s_n\}]$.
Lemma \ref{l2} implies $\xi[s,s_n)+v_n\le\eta[s,s_n]$ 
and hence $\xi[s,s_n)\le \eta[s,s_n]$.
Since $t^*<t_n=\zeta(s_n+v_n)$ we obtain from
\eqref{e3.91} and Lemma \ref{l1} that $s_n>\zeta^{-1}(t^*)=s$.
Lemma \ref{l1} and the continuity of $\zeta^{-1}$ imply
$s_n=\zeta^{-1}(t_n)\downarrow \zeta^{-1}(t^*)=s$ along the chosen
subsequence. Hence $\tau^0(s)=s$.
In the second case, there are infinitely many $n\in\N$ such
that $t_n=\zeta(s_n)+v_n$ for some $s_n\in\R$ satisfying
$\eta\{s_n\}>0$ and $v_n\in[0,\eta\{s_n\}]$.
Then $\xi\{s_n\}=0$ and
Lemma \ref{l2} implies $\xi[s,s_n)\le \eta[s,s_n)+v_n$ 
and hence $\xi[s,s_n)\le \eta[s,s_n]$.
As before, it follows that $s_n>s$
and $s_n\downarrow s$ along the chosen
subsequence. Hence $\tau^0(s)=s$ in this case.

%In the remainder of the proof, we assume that $t^*>\zeta(s)$.  
Assume next that $t^*\in(\zeta(s),\infty)$.
By definition,
\begin{align}\label{e3.12}
\xi^*[\zeta(s),r]>\eta^*[\zeta(s),r],\quad \zeta(s)<r<t^*,
\end{align}
as well as
\begin{align}\label{e3.14}
  \xi^*[\zeta(s),t^*]=\eta^*[\zeta(s),t^*].
\end{align}
Assume first that $t^*=\zeta(t)+v$ for some $t\ge s$
with $\eta\{t\}=0$ and $v\in[0,\xi\{t\}]$. Then $t>s$ (by \eqref{e3.14}) and
Lemma \ref{l1} implies that $\zeta^{-1}(t^*)=t$. We want to show
that $\tau^0(s)=t$.
Let $t'\in (s,t)$. If $\xi\{t'\}=0$, we set 
$r:=\zeta(t')+\eta\{t'\}$. Then $r<\zeta(t)\le t^*$
and \eqref{e3.12} together with Lemma \ref{l2} imply that
$\xi[s,t')>\eta[s,t']$.
If $\xi\{t'\}>0$, we set $r:=\zeta(t')$ to obtain
the same inequality and hence
\begin{align}\label{e3.16}
\xi[s,t')>\eta[s,t'],\quad s<t'<t.
\end{align}
On the other hand, we have from \eqref{e3.14}
and Lemma \ref{l2} that $\xi[s,t)+v=\eta[s,t]$,
so that $\xi[s,t)\le\eta[s,t]$.
Hence $\tau^0(s)=t=\zeta^{-1}(t^*)$ and \eqref{e3.10}
follows.

The second possible case is 
$t^*=\zeta(t)+v$ for some $t\ge s$
with $\eta\{t\}>0$ and $v\in[0,\eta\{t\}]$.
Since $\xi$ and $\eta$ are mutually singular, we
have $\xi\{t\}=0$.
Again this implies $t>s$ and \eqref{e3.16}.
Lemma \ref{l2} implies that $\xi[s,t)+v=\eta[s,t)$
and hence $\xi[s,t)\le\eta[s,t]$. Therefore
$\tau^0(s)=t=\zeta^{-1}(t^*)$, where we have used Lemma \ref{l1}.

Assume, finally, that $t^*=\infty$, so that \eqref{e3.12}
holds for all $r>\zeta(s)$.
Let $t'>s$. If $\xi\{t'\}=\eta\{t'\}=0$,
we take $r=\zeta(t')$ to obtain from Lemma \ref{l2} that
$\xi[s,t')>\eta[s,t']$.
If $\xi\{t'\}>0$ (and hence $\eta\{t'\}=0$),
we take $r=\zeta(t')$ to obtain from Lemma \ref{l2} that
$\xi[s,t')>\eta[s,t']$.
If $\eta\{t'\}>0$ (and hence $\xi\{t'\}=0$),
we take $r=\zeta(t')+\eta\{t'\}$ to obtain from Lemma \ref{l2} that
$\xi[s,t')>\eta[s,t']$. Hence $\tau^0(s)=\infty$. 
\end{proof}

\begin{lemma}\label{l14} Let $s\in\R$ with 
$\xi\{s\}>0$ and $u\in[0,1)$. Then
$\tau^*(\zeta(s)+u\xi\{s\})<\infty$ iff $\tau^{1-u}(s)<\infty$.
In this case
\begin{align}\label{e3.21}
\zeta^{-1}(\tau^*(\zeta(s)+u\xi\{s\}))=\tau^{1-u}(s).
\end{align}
\end{lemma}
\begin{proof}
 Since $\xi$ and $\eta$ are mutually singular
we have $\eta\{s\}=0$. Moreover,
since $u<1$ we have $\eta[s,s+\varepsilon]<(1-u)\xi\{s\}$ 
for all sufficiently small $\varepsilon>0$ and therefore
$\tau^{1-u}(s)>s$ and (by Lemma \ref{l2})
$\tau^*(\zeta(s)+u\xi\{s\})>\zeta(s)$.
The proof can now proceed similar to that of Lemma \ref{l13}.
The main tool is again Lemma \ref{l2}. In contrast
to the case $\xi\{s\}=0$, it has to be applied
with $s_1=s$ and $v_1=u\xi\{s\}$. Further details
are omitted.
\end{proof}

In the upcoming proof of Theorem \ref{tmain1}, we will use that $\tau^*$ balances $\xi^*$ and $\eta^*$. 
Since $\xi^*$ and $\eta^*$ need not be jointly stationary, Theorem \ref{LMT5.1} 
cannot be used directly to establish this fact. As an intermediate step, we need the 
following lemma which presents (shifted and length-biased) versions of $\xi^*$ and $\eta^*$ 
that Theorem \ref{LMT5.1} can be applied to. 
As before we abuse notation by introducing
for any measure $\mu$ on $\R$ and $s\in\R$ a new measure
$\theta_s\mu$ by $\theta_s\mu:=\mu(\cdot +s)$.
If $\mu'$ is another measure on $\R$, we write
$\theta_s(\mu,\mu'):=(\theta_s\mu,\theta_s\mu')$.

\begin{lemma}\label{stat}
Let $\xi$ and $\eta$ be random measures on $\R$ defined on 
$(\Omega,\mathcal{F},\BP)$ and let $\xi^*$ 
and $\eta^*$ be as above. Extend  $(\Omega,\mathcal{F},\BP)$ so as to 
support a random variable $U$
that is uniform on $[0,1)$ conditionally on $(\xi^*,\eta^*)$. 
Let $S_1 = 1 + \xi^d[0,1)+\eta^d[0,1)$ and define a $\sigma$-finite measure 
$\BP^*$ on $(\Omega,\mathcal{F})$ by 
$d\BP^* = S_1d\BP$. Then the distribution of
$\theta_{US_1}(\xi^*,\eta^*)$ is stationary and ergodic
under $\BP^*$ and the intensities are 
\begin{align*}
\BE^*(\theta_{US_1}\xi^*)[0,1) = \BE\xi[0,1),\quad 
%\text{and} \quad
\BE^*(\theta_{US_1}\eta^*)[0,1) = \BE\eta[0,1).
\end{align*}
%$\theta_{US_1}\eta^*$ 
%are orthogonal with have the same finite intensity as $\xi$ and $\eta$ have under $\BP$.
\end{lemma}
\begin{proof} 
Let $M$ denote the space of locally finite measures on $\R$, equipped with
the natural Kolmogorov $\sigma$-field. We have assumed that $\xi$ and $\eta$ are random elements
of $M$.

Take $t \geq 0$ and let $f\colon M\times M\to[0,\infty)$ be bounded 
and measurable. % and set $f_n = f1_{A_n}$.
For $n \in \mathbb{N}$, let $A_n \in \mathcal{F}$
%Let $A_1 \subseteq A_2  \subseteq A_2 \subseteq \dots\uparrow \Omega$ be events 
be such that 
$\BP(A_n) < \infty$ and $A_1 \subset A_2  \subset A_3 \subset \dots\uparrow \Omega$. 
Set $f_n = f1_{A_n}$ and note that $\BP(A_n) < \infty$ implies that 
$\BE \int_{S_1}^{S_1+t} f_n(\theta_s(\xi^*,\eta^*))ds$ is finite. 
Since $\int_t^{S_1} f_n(\theta_s(\xi^*, \eta^*))ds$ could be negative, the 
finiteness is needed for the last equality in
\begin{align*}
\BE^*f_n(\theta_t\theta_{US_1}(\xi^*, \eta^*))&=\BE S_1 f_n(\theta_t\theta_{US_1}(\xi^*, \eta^*))
=\BE \int_t^{S_1+t} f_n(\theta_s(\xi^*, \eta^*))ds\\
&=\BE \int_t^{S_1} f_n(\theta_s(\xi^*, \eta^*))ds + \BE \int_{S_1}^{S_1+t} f_n(\theta_s(\xi^*, \eta^*))ds.
\end{align*}
Now note that $\theta_{S_1}(\xi^*,\eta^*)$ is the same measurable mapping of 
$\theta_1(\xi,\eta)$ as $(\xi^*,\eta^*)$ is of 
$(\xi,\eta)$. 
Since $\theta_1(\xi,\eta)$ has the same  $\BP$-distribution as $(\xi,\eta)$, 
this implies that 
$\theta_{S_1}(\xi^*,\eta^*)$ has the same $\BP$-distribution as $(\xi^*,\eta^*)$. 
This further yields the first equality in  
%for $t \geq 0$ and boundend nonnegative measurable $f$,
\begin{align*}
\BE^*f_n(\theta_t\theta_{US_1}(\xi^*, \eta^*))
&=\BE \int_t^{S_1} f_n(\theta_s(\xi^*, \eta^*))ds + \BE \int_0^{t} f_n(\theta_s(\xi^*, \eta^*))ds\\
&=\BE \int_0^{S_1} f_n(\theta_s(\xi^*, \eta^*))ds.
\end{align*}
Send $n\to\infty$ and use the monotone convergence theorem to obtain
\begin{align*}
\BE^*f(\theta_t\theta_{US_1}(\xi^*, \eta^*))=\BE \int_0^{S_1} f(\theta_s(\xi^*, \eta^*))ds.
\end{align*}
So $\BE^*f(\theta_t\theta_{US_1}(\xi^*, \eta^*))$ does not depend on $t$; that is, 
$\theta_{US_1}(\xi^*,\eta^*)$ is stationary under $\BP^*$.

Next we prove the ergodicity assertion.  
It is not hard to prove that
\begin{align}\label{anotherlemma}
((\theta_t\xi)^*,(\theta_t\eta)^*)=(\theta_{\zeta(t)}\xi^*,\theta_{\zeta(t)}\eta^*),\quad t\in\R.
\end{align}
Let $A\subset M\times M$ be a measurable set that is invariant under
diagonal shifts. Then \eqref{anotherlemma} shows for all $(\omega,t)\in\Omega\times\R$
that $(\xi(\theta_t\omega)^*,\eta(\theta_t\omega)^*)\in A$ if and only if
$(\xi(\omega)^*,\eta(\omega)^*)\in A$. Since $\BP$ is ergodic, we obtain
that either $\BP((\xi^*,\eta^*)\in A)=0$ or $\BP((\xi^*,\eta^*)\notin A)=0$.
Therefore, the ergodicity of 
the shifted length-biased version of $(\xi^*,\eta^*)$ 
follows from the two facts that the zero sets of $\BP^*$ 
are the same as those of $\BP$, and that randomly shifted invariant sets remain invariant.

It remains to prove the intensity result. Let 
$\mu_1$ be Lebesgue measure on $\R$.  
Since $\mu_1$ is shift-invariant and 
$\theta_{S_1}\xi^*$ has the same $\BP$-distribution as $\xi^*$, we have
\begin{align*}
\BE(\mu_1 \otimes \xi^*)(A) = \BE(\mu_1 \otimes \xi^*)((S_1, S_1) + A),
\quad A \in \mathcal{B}(\R^2).
\end{align*}

\begin{figure}[h]
\includegraphics[width=0.8 \textwidth]{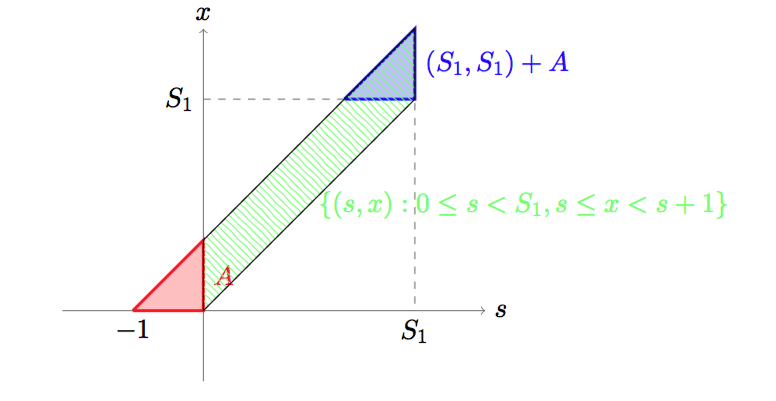}
\caption{The region $\{(s,x) : 0\leq s < S_1, s \leq x < s+1\}$ 
has the same Lebesgue measure as $(\{(s,x) : 0\leq x < S_1, x - 1 < s \le x\}
\setminus A) \cup ((S_1, S_1)+A)$.}
\end{figure}
 
Apply this with   $A = \{(s,x) : -1 \leq s < 0, 0\leq x < s+1 \}$ to obtain
\begin{align*}
\BE^*(&\theta_{US_1}\xi^*)[0,1) = \BE S_1 \int_0^1( \theta_{uS_1}\xi^*)[0,1)du 
= \BE \int_0^{S_1} \xi^*(s+[0,1))ds\\
%= \BE \int_0^{S_1}\int_{[0,1)} \xi^*(s+dx)ds
&= \BE (\mu_1 \otimes \xi^*)(\{(s,x) : 0\leq s < S_1, s \leq x < s+1\})\\
&= \BE (\mu_1 \otimes \xi^*)((\{(s,x) : 0\leq x < S_1, x - 1 < s \le x\}
\setminus A) \cup ((S_1, S_1)+A))\\
&= \BE (\mu_1 \otimes \xi^*)(\{(s,x) : 0\leq x < S_1, x - 1 < s \le x\})\\
&\qquad\qquad\qquad\qquad\qquad\qquad
- \BE (\mu_1 \otimes \xi^*)(A)  + 
\BE (\mu_1 \otimes \xi^*)((S_1, S_1)+A))\\
&= \BE (\mu_1 \otimes \xi^*)(\{(s,x) : 0\leq x<S_1, x-1<s<x\})\\
&= \BE\xi^*[0,S_1)= \BE\xi[0,1).
\end{align*}
In the same way, we obtain $\BE^*(\theta_{US_1}\eta^*)[0,1) = \BE\eta[0,1)$.
\end{proof}

\begin{proof}[Proof of Theorem \ref{tmain1}]
By Theorem \ref{LMT5.1} and Lemma \ref{stat}, the random mapping
$\tau'\colon\R\to[0,\infty]$ defined by
\begin{align*}
\tau'(s):=\inf\{t> s \colon (\theta_{US_1}\xi^*)[s,t]
\le (\theta_{US_1}\eta^*)[s,t]\},\quad s\in\R,
\end{align*}
balances $\theta_{US_1}\xi^*$ and $\theta_{US_1}\eta^*$ under $\BP^*$. 
Since  $\BP^*$ and $\BP$ are equivalent, this implies that 
$\tau'$ balances  $\theta_{US_1}\xi^*$ and $\theta_{US_1}\eta^*$ under $\BP$. 
This remains true if the origin is shifted to some random location.
Recall the definition \eqref{matching*} of $\tau^*$. Shift the origin to $-US_1$ to obtain 
\begin{align*}
\tau^*=\tau'(\cdot-US_1)+US_1,\qquad 
\xi^* = \theta_{-US_1}\theta_{US_1}\xi^*,\qquad
\eta^* = \theta_{-US_1}\theta_{US_1}\eta^*.
\end{align*} 
Thus, $\tau^*$ balances $\xi^*$ and $\eta^*$ under $\BP$. We assume for simplicity
that
\begin{align}\label{*taubalancing}
\int \I\{\tau^*(s)\in\cdot\}\,\xi^*(ds)=\eta^*
\end{align}
holds everywhere on $\Omega$.

%We assume now for simplicity that
%$\tau^*$ balances $\xi^*$ and $\eta^*$ everywhere on $\Omega$. 
We want to prove that
\begin{align}\label{e3.100}
\iint f(t)\,K(s,dt)\,\xi(ds)=\int f(s)\,\eta(ds)
\end{align}
for all measurable $f\colon\R\to[0,\infty)$, implying the theorem.
Applying Lemma \ref{l3} with $\eta$ in place of $\xi$ and then
the balancing property \eqref{*taubalancing} of $\tau^*$ yields
\begin{align*}
\int &f(t)\,\eta(dt)=\int f(\zeta^{-1}(t))\,\eta^*(dt)
=\int \I\{\tau^*(s)<\infty\}f(\zeta^{-1}(\tau^*(s)))\,\xi^*(ds)=I_1+I_2,
\end{align*}
where
\begin{align*}
I_1:=&\int \I\{\tau^*(\zeta(s))<\infty\}f(\zeta^{-1}(\tau^*(\zeta(s)))\,
\xi^c(ds),\\
I_2:=&\iint \I\{0\le u\le 1\}\I\{\tau^*(\zeta(s)+u\xi\{s\}))<\infty\}
f(\zeta^{-1}(\tau^*(\zeta(s)+u\xi\{s\})\,du\,\xi^d(ds).
\end{align*}
Here we have used the definition of $\xi^*$ and a change of variables in
the second summand.

Lemma \ref{l13} implies that
$%$
I_1=\int \I\{\tau^0(s)<\infty\}f(\tau^0(s))\,\xi^c(ds).
$ %$
Note that $\tau^u(s)=\tau^0(s)$ for all $u\in[0,1]$ whenever
$\xi\{s\}=0$.  It follows that
$$
I_1=\iint f(t)K(s,dt)\,\xi^c(ds).
$$
To treat $I_2$, we use Lemma \ref{l14} to obtain that
\begin{align*}
I_2&=
\iint \I\{0\le u<1\}\I\{\tau^{1-u}(s)<\infty\}f(\tau^{1-u}(s))\,du\,\xi^d(ds)%\\&
=\iint f(t)K(s,dt)\,\xi^d(ds),
\end{align*}
which proves \eqref{e3.100}.
\end{proof}

\section{On Palm measures and shift-coupling}\label{Palm}

Let  $\xi$ be an invariant random measure on $\R$.
The {\em Palm measure} $\BP_\xi$ of $\xi$
(with respect to $\BP$) is defined by
\begin{align}\label{Palmmeasure} 
\BP_\xi(A):=\BE \int \I_{[0,1]}(s)\I_A(\theta_s)\, \xi(ds),
\quad A\in\mathcal{F}.
\end{align}
This is a $\sigma$-finite measure on $(\Omega,\mathcal{F})$
satisfying the {\em refined Campbell formula}
\begin{align}\label{refC}
\iint f(\theta_s\omega,s)\,\xi(\omega,ds)\, \BP(d\omega)=
\iint f(\omega,s)\,ds\,\BP_\xi(d\omega)
\end{align}
for each measurable $f\colon\Omega\times\R\to [0,\infty)$;
see e.g.\ \cite[Chapter 11]{Kallenberg} and \cite{LaTho09}.
If the {\em intensity} $\BP_\xi(\Omega)$ of $\xi$ is
positive and finite, then $\BP_\xi$ can be normalized to
yield the Palm probability measure of $\xi$. 
This normalization can be interpreted as conditional version of $\BP$ given that
the origin represents a point randomly chosen in the mass
of $\xi$; see \cite{LaTho09,LaTho11}.
%This measure can be interpreted as the conditional distribution
%(with respect to $\BP$) given that the origin $0\in\R$
%is a {\em typical point} in the mass of $\xi$,
%see \cite[Chapter 11]{Kallenberg} for some fundamental properties
%of Palm probability measures.

The following {\em shift-coupling} result is a consequence of Theorem \ref{tmain1}
and \cite[Theorem 4.1]{LaTho09}.

\begin{proposition}\label{lstt}
Let $\xi$ and $\eta$ satisfy the assumptions of Theorem \ref{tmain1}
and define the allocation maps $\tau^u$, $u\in[0,1]$, by \eqref{umatching}.
Let $T^u:=\tau^u(\cdot,0)$. Then
\begin{align}\label{shift-coupling1} 
\int^1_0\BP_\xi(\theta_{T^u}\in\cdot)\,du=\BP_\eta,
\end{align}
where $\theta_{T^u}\colon\Omega\to\Omega$ is defined by 
$\theta_{T^u}(\omega):=\theta_{{T^u}(\omega)}\omega$, $\omega\in\Omega$.
If, moreover, $\xi$ is diffuse then
\begin{align}\label{shift-coupling2}
\BP_\xi(\theta_{T}\in\cdot)=\BP_\eta,
\end{align}
where $T:=\tau(\cdot,0)$ and the allocation $\tau$ is defined by \eqref{matching}.
\end{proposition}

\begin{remark}\label{U}\rm
If we extend $(\Omega,\mathcal{F},\BP)$ to support an independent 
random variable $U$ uniformly distributed on $[0,1]$, then 
$\tau^U$ is a {\em randomized} allocation balancing $\xi$
and $\eta$ even when $\xi$ is not diffuse. 
Moreover, \eqref{shift-coupling1} can then be written in the 
same shift-coupling form as \eqref{shift-coupling2}, namely
\begin{align*}
\BP_\xi(\theta_{T^U}\in\cdot)=\BP_\eta.
\end{align*}
\end{remark}

In Brownian excursion theory, it is natural to define Palm measures
of random measures that are not locally finite.
A {\em $\sigma$-finite random measure} $\xi$ is
a kernel from $\Omega$ to $\R$ with the following
property. There exist measurable sets $A_n\in\mathcal{F}$, $n\in\N$,
such that $\cup_n A_n=\Omega$ and
\begin{align}\label{An}
\int \I\{s\in\cdot,\theta_s\in A_n\}\,\xi(ds)
\end{align}
is a random measure for each $n\in\N$.
In this case, the Palm measure $\BP_\xi$ can again be defined by \eqref{Palmmeasure}.
It is $\sigma$-finite and satisfies the refined Campbell formula \eqref{refC};
see \cite{Pitman87} for a special case.
We shall use such a measure in Proposition \ref{propPitman}.

\section{Excursions of Brownian motion}\label{firstexcursion}

In the next two sections, we assume that 
$\Omega$ is the class of all continuous functions $\omega\colon\R\rightarrow\R$
equipped with  the Kolmogorov product $\sigma$-algebra $\mathcal{F}$.
Let $B=(B_t)_{t\in\R}$ denote the identity on $\Omega$.
The flow is given by 
\begin{align}\label{shiftB}
(\theta_t\omega)_s:=\omega_{t+s}.
\end{align}
Let $\BP_0$ denote the  distribution of a two-sided standard Brownian motion. 
Define  $\BP_x:=\BP_0(B+x\in\cdot)$, $x\in\R$, and
the $\sigma$-finite and stationary measure 
\begin{align}\label{statP}
\BP:=\int \BP_x \, dx.
\end{align}
By \cite[Theorem 3.5]{LaMoeTho12} this $\BP$ is ergodic.
Expectations with respect to $\BP_x$ and $\BP$ are denoted
by $\BE_x$ and $\BE_\BP$, respectively. 

Let $t\in\R$. For each real-valued function $w$ whose domain contains $[t,\infty)$,
let
\begin{align*}
D_t(w):=\inf\{s>t: w(s)=0\},
\end{align*}
where $\inf\emptyset:=\infty$. We abbreviate $D(w):=D_0(w)$. Then
$$
R_t:=D(\theta_tB)=D_t(B)-t
$$
is the time taken by $B$ to hit $0$ (starting at time $t$), while 
\begin{align*}
L:=\{t \in\R:R_{t-}=0, R_t>0\},
\end{align*}
is the set of left ends of {\em excursion intervals}. 
The space $E$ of excursions is
the class of all continuous functions $e\colon [0,\infty)\rightarrow\R$
such that $e(0)=0$, $0<D(e)<\infty$, and $e(t)=0$ for all $t\ge D(e)$.
The number $D(e)$ is called the lifetime of the excursion.
We equip $E$ with the Kolmogorov product $\sigma$-field $\mathcal{E}$.
For $s\in L$, define the (random) excursion $\epsilon_s\in E$ 
starting at time $s$ by
\begin{align*}
\epsilon_{s}(t):= 
\begin{cases}
B_{s+t},& \text{if $0\le t \le R_s$},\\
0,&\text{if $t>R_s$}.
\end{cases}            
\end{align*}
It is convenient to introduce the function $\delta\equiv 0$ on $[0,\infty)$,
and to define $\epsilon_s:=\delta$
for $s\notin L$. Then $(\omega,t)\mapsto \epsilon_t(\omega)$
is a measurable mapping with values  in $E_\delta:=E\cup\{\delta\}$.
Note that $D(\delta)=0$.

Define a $\sigma$-finite invariant random measure $N$ on $\R$ by
\begin{align}\label{N}
N(C):=\sum_{s\in L}\I\{s\in C\}, \quad C\in\cB(\R).
\end{align}
Here the invariance is obvious, while we may choose
$A_1:=\{D\in \{0,\infty\}\}$ and $A_n:=\{D\ge 1/n\}$, $n\ge 2$,
in \eqref{An}, to see that $N$ is $\sigma$-finite.
It follows from the refined Campbell formula \eqref{refC} that
\begin{align*}
\BE\int\I\{(s,\epsilon_s)\in\cdot\}\,N(ds)=\iint \I\{(s,e)\in\cdot\}\,ds\,\nu(de),
\end{align*}
where $\nu:=\BP_N(\epsilon_0\in\cdot)$, that is
\begin{align}\label{nu}
\nu(A)=\BE\int_{[0,1]}\I\{\epsilon_s\in A\}\,N(ds),\quad A\in\mathcal{E}.
\end{align}
(Note that $\nu\{\delta\}=0$.)
In fact, Pitman \cite{Pitman87} showed that
$\nu$ coincides with {\em It\^{o}'s excursion measure} (suitably normalized).

%In order to prove Proposition \ref{propPitman} and then to apply 
%Proposition \ref{lstt}, we use the local times of $B$.

For $x\in\R$, we denote by $\ell^x$ the random measure
associated with the {\em local time} of $B$ 
at $x\in\R$ (under $\BP_0$). 
%This means that
%\begin{align}\label{occupation}
%\int f(B_s,s)\, ds=\iint f(x,s)\,\ell^x(ds)\, dx\quad \BP_0\text{-a.s.}
%\end{align}
%for all measurable $f\colon\R^2\rightarrow[0,\infty)$.
The global construction in \cite{Perkins81}
(see also \cite[Proposition 22.12]{Kallenberg} 
and \cite[Theorem~6.43]{MortersPeres})
guarantees the existence of a version of local times
with the following properties.
The random measure $\ell^0$ is $\BP_x$-a.e.\ diffuse for each $x\in\R$ and  
\begin{align}\label{0inv}
&\ell^0(\theta_t\omega,C-t)=\ell^0(\omega,C),\quad C\in\mathcal{B}(\mathbb{R}),\,t\in\R,\,
\BP_x\text{-a.s.},\,x\in\R,\\
\label{0x}
&\ell^y(\omega,\cdot)=\ell^0(\omega-y,\cdot),\quad \omega\in\Omega,\, y\in\R,\\
\label{0y}
&\supp\ell^x(\omega) =\{t\in\R:B_t=x\},\quad  \omega\in\Omega,\,x\in\R,
\end{align}
where $\supp\mu$ is the support of a measure $\mu$ on $\R$.
Equation \eqref{0x} implies that $\ell^y$ is $\BP_x$-a.e.\ diffuse for 
each $x\in\R$ and is invariant in the sense of \eqref{0inv}.
By a classical result from \cite{GeHo73} (see also \cite[Lemma 2.3]{LaMoeTho12}),
the Palm measure of $\ell^x$ is given by
\begin{align}\label{palmlocaltime}
\BP_{\ell^x}=\BP_x,\quad x\in\R.
\end{align}

For $t\ge 0$, let $\ell_t^0:=\ell^0([0,t])$.
Define the right inverse of
$\ell^0$ by $\tau_s:=\inf\{t\ge 0: \ell^0_t>s\}$ for
$s\ge 0$. By \eqref{0y},
\begin{align*}
\{t\ge 0:B_t\ne 0\}=\bigcup_{s>0} (\tau_{s-}, \tau_s).
\end{align*}
%where $D$ is the countable set of discontinuity points of $s\mapsto \tau_s$. 
If $\tau_{s-}<\tau_s$, then $(\tau_{s-},\tau_s)$ is an excursion
interval away from $0$. 
A classical result of It\^{o} \cite{Ito71} (see also \cite[Theorem 22.11]{Kallenberg}
and \cite[Theorem XII(2.4)]{RevuzYor99}) shows that the random measure
\begin{align}\label{Poissonexcursion}
\Phi:=\sum_{s> 0:\tau_{s-}<\tau_s}\delta_{(s,\epsilon_{\tau_{s-}})}
\end{align}
is a Poisson process on $(0,\infty)\times E$ under $\BP_0$ with intensity measure
\begin{align*}
\BE_0\sum_{s> 0:\tau_{s-}<\tau_s}\I\{(s,\epsilon_{\tau_{s-}})\in\cdot\}
=\int_E\int^\infty_0 \I\{(s,e)\in\cdot\}\,ds\,\nu(de).
\end{align*}
The excursion measure $\nu$ satisfies
\begin{align}\label{e11}
\nu(D\in dr)=cr^{-3/2}\,dr
\end{align}
for some constant $c>0$;
see \cite[Section XII.2]{RevuzYor99} or \cite[Theorem 22.5]{Kallenberg}.

In the next section, we return to the problem discussed in the introduction
of finding an {\em extra} $A$-{\em excursion}.
Before embarking on this
by means of Palm and transport theory, we check what happens 
if we simply choose for a given $A\in\mathcal{E}$ the first excursion 
belonging to $A$ to the right of the origin. 
For a measure $\mu$ and a set $A$ such that $0<\mu(A)<\infty$, 
we define the conditional measure $\mu(\cdot \mid A)=\mu(\cdot \cap A)/\mu(A)$.
The following well-known result can be derived with the help of
excursion theory; see \cite[Lemma XII(1.13)]{RevuzYor99}.
It is a special case of \eqref{e6.2}, to be proved below.

\begin{proposition} \label{naiveembedding}%\cite{RevuzYor99} 
Let $A\in\mathcal{E}$ satisfy $0<\nu(A)<\infty$ and define
$$
S_A:=\inf\{t>0:t\in L,\epsilon_t\in A\}.
$$
Then $\BP_0(\epsilon_{S_A}\in \cdot)=\nu(\cdot \mid  A)$.
\end{proposition}

\begin{figure}[h]
\hspace{0.8cm}
\includegraphics[width=0.9 \textwidth]{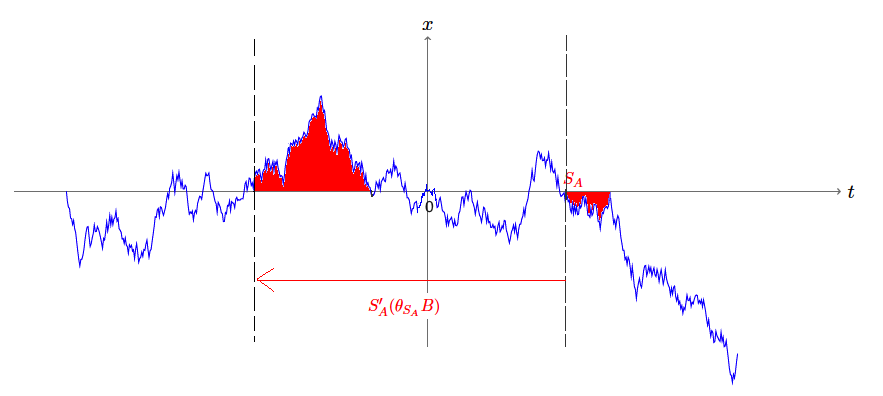}
\caption{The times $S_A$ and $S_A^{'}(\theta_{S_A}B)$ in a two-sided Brownian motion.}
\end{figure}

Thus, $\epsilon_{S_A}$ is an $A$-excursion. Also, with $X_A$ the length of 
$\epsilon_{S_A}$,
an independent standard Brownian motion $((\theta_{S_A+X_A}B)_s)_{s\ge 0}$
starts at its right endpoint. However, the process
$((\theta_{S_A}B)_{-s})_{s\ge 0}$ is not a standard Brownian motion because it
starts with a path segment of positive length $S_A$ without an
$A$-excursion and then at time $S_A$ an independent standard Brownian
motion starts. 
In fact, it follows from the Poisson nature of the point process \eqref{Poissonexcursion}
that both $\ell^0([0,S_A])$ and $\ell^0([S'_A,0])$ have (under $\BP_0$) an exponential
distribution with rate parameter $\nu(A)$, 
where $S'_A:=\sup\{t<0:t\in L,\epsilon_t\in A\}$.
Since these random variables are independent,
it follows that the local time accumulated by $\theta_{S_A}B$
on the interval $[S'_A(\theta_{S_A}B),0]$
has a Gamma distribution with shape parameter $2$.
Hence the embedding of the conditional It\^{o} measure $\nu(\cdot \mid  A)$ in 
Proposition~\ref{naiveembedding} is not {\em unbiased}, that is, $\epsilon_{S_A}$ 
is not an {\em extra} $A$-excursion.

%The above discussion shows that a two-sided standard Brownian
%motion has the structure described in the introduction: 
%it splits a.s.\ into a two-sided
%sequence of independent segments  such that: every odd-numbered
%segment is an $A$-excursion;  every even-numbered segment except the one enumerated
%$0$ is a standard Brownian motion starting from zero running until the first time that
%an $A$-excursion occurs; the segment enumerated $0$
%consists of two independent segments of that type; the origin is placed between 
%those two segments.   In addition to
%this, the origin of the standard Brownian motion should be placed at
%random in the segment enumerated $0$ according to the local time at
%zero.
% 
%This structure finally suggests that since the Palm version with
%respect to the point process $N_A$, formed by the times when 
%excursions belonging to $A$ begin, is point-stationary, it should have
%the above structure except that the segment enumerated $0$ should be
%like the other even-numbered segments and should end at time~$0$. In
%particular, this would imply that under $\BP_{N_A}$, the backward
%Brownian motion $(B_t)_{t \leq 0}$ and the excursion $\epsilon_0$
%beginning at $0$ are independent. 
%%a property that has already been
%%used in the proof of Proposition \ref{naiveembedding}.
 
\section{Finding an extra excursion}\label{extra}

Let $A$ be a measurable set of excursions with positive and
finite It\^{o} measure. In this
section, we will use Proposition~\ref{lstt}  (with $\xi = \ell^0$) for unbiased 
embedding of an $A$-excursion; see Theorem \ref{tpathdeito}.
For this purpose, we need to show that 
under the Palm measure of $N$ (see \eqref{N}), the 
Brownian motion is decomposed into three independent 
pieces: a time reversed Brownian motion on $(-\infty,0]$, an excursion
with `distribution' $\nu$, and a Brownian motion starting after this
excursion.  A formal statement of this result requires some notation.

Let $\Omega^-$ (resp.\ $\Omega^+$) denote the space of all continuous
functions $w$ on $(-\infty,0]$ (resp.\ $[0,\infty)$) with $w(0)=0$. The
{\em concatenation} of $w_1\in\Omega^-$,$w_2\in E$ and $w_3\in\Omega^+$ 
is the function $w_1\odot w_2\odot w_3\in\Omega$ defined by
\begin{align*}
w_1\odot w_2\odot w_3(t):=
\begin{cases}
w_1(t),& \text{if $t \le 0$},\\
w_2(t),&\text{if $0<t<D(w_2)$},\\
w_3(t-D(w_2)),&\text{if $D(w_2)\le t$}.
\end{cases}            
\end{align*}
The {\em concatenation} of $\sigma$-finite measures $\nu_1$ on $\Omega^-$, $\nu_2$ on  $E$, 
and $\nu_3$ on $\Omega^+$
is the measure  $\nu_1\odot \nu_2\odot \nu_3$ on $\Omega$ defined by
\begin{align*}
\nu_1\odot \nu_2\odot \nu_3:=\iiint \I\{w_1\odot w_2\odot w_3\in\cdot\}\,
\nu_1(dw_1)\,\nu_2(dw_2)\,\nu_3(dw_3).
\end{align*}
Let $\BP^-$ (resp.\ $\BP^+$) denote the law of $B^-:=(B_t)_{t\le 0}$ 
(resp.\ $B^+:=(B_t)_{t\ge 0}$).

The following proposition will be proved below. It extends a result from Pitman \cite{Pitman87}.

\begin{proposition}\label{propPitman} 
The Palm measure of $N$ is given by $\BP_N=\BP^-\odot \nu \odot \BP^+$. 
\end{proposition}

Let $A\in\mathcal{E}$ be such that $0<\nu(A)<\infty$. Define
an invariant random measure $N_A$ on $\R$ by
\begin{align}\label{NA}
N_A(C):=\frac{1}{\nu(A)}\int\I\{s\in C,\epsilon_s\in A\}\,N(ds), \quad C\in\cB(\R),
\end{align}
We then have the following immediate consequence of Proposition \ref{propPitman}.

\begin{corollary}\label{c12} Let $A\in\mathcal{E}$ satisfy $0<\nu(A)<\infty$
and define the random measure $N_A$ by \eqref{NA}. 
Then the Palm measure of $N_A$ is given by $\BP_{N_A}=\BP^-\odot \nu(\cdot|A) \odot \BP^+$.
\end{corollary}

Recall the definition 
$S'_A:=\sup\{t<0:t\in L,\epsilon_t\in A\}$, $A\in\mathcal{E}$.

\begin{remark}\label{pointstat}\rm
Let $A\in\mathcal{E}$ satisfy $0<\nu(A)<\infty$. Under its 
Palm measure $\BP_{N_A}$, the random measure $N_A$
is {\em point-stationary}, that is distributionally invariant
under the (random) shifts $\theta_{S_A}$ and $\theta_{S'_A}$; 
see \cite{Thor96,Thor00,LaTho09}.
\end{remark}

Together with Corollary \ref{c12} and Remark \ref{pointstat} the following result
can be used to describe the splitting of a Brownian motion
into independent segments (see the introduction) in a
rigorous manner. 
\begin{proposition}\label{pneveu} Let $A$ be as in Corollary \ref{c12}
and suppose that $f\colon \Omega\to[0,\infty)$ is measurable.
Then
\begin{align*}
\BE_0 f=\nu(A)\BE_{N_A}\int f\circ\theta_t \I\{S'_A\le t\le 0\}\,\ell^0(dt). 
\end{align*}
\end{proposition}
\begin{proof} By \eqref{palmlocaltime} we have $\BP_{\ell^0}=\BP_0$.
Therefore, taking a measurable $h\colon\Omega\times\R\to[0,\infty)$,
we obtain from Neveu's exchange formula (see e.g.\ \cite{LaTho09}) that
\begin{align*}
\BE_0\int h(\theta_0,t)\,N_A(dt)=\BE_{N_A}\int h(\theta_t,-t)\,\ell^0(dt).
\end{align*}
We apply this formula
with $h(\theta_0,t)=\I\{0<t\le S_A\}f$ to obtain that
\begin{align*}
\nu(A)^{-1}\BE_0 f=\BE_{N_A}\int f(\theta_t)\I\{0\le -t\le S_A\circ\theta_t\}\,\ell^0(dt).
\end{align*}
It remains to note that $0\le -t\le S_A\circ\theta_t$ iff $S'_A\le t\le 0$.
\end{proof}

\begin{remark}\label{rfirstexcursion}\rm
Let $g\colon \Omega^-\to[0,\infty)$ and $h\colon E\to[0,\infty)$
be measurable functions.
Combining Proposition \ref{pneveu} with Corollary \ref{c12}
shows after a short calculation that
\begin{align}\label{e6.2}
\BE_0 g((\theta_{S_A}B)^-)h(\epsilon_{S_A})
=\nu(A)\BE_0\big[\ell^0((S'_A,0])g(B^-)\big]\int h(e)\,\nu(de\mid A).
\end{align}
In particular, $(\theta_{S_A}B)^-$ and $\epsilon_{S_A}$ are independent, as asserted
in the introduction. Moreover, $\epsilon_{S_A}$ has distribution $\nu(\cdot \mid A)$,
as asserted by Proposition \ref{naiveembedding}.
\end{remark}

The following path decomposition of a two-sided Brownian motion
is the main result of this section.

\begin{theorem}\label{tpathdeito}
Let $A \in \mathcal{E}$ be such that $0<\nu(A)<\infty$
and define the random measure $N_A$ by \eqref{NA}. Let
\begin{align*}
T:=\inf\{t>0: \ell^0[0,t]\le N_A[0,t]\}.
\end{align*}
Then $\BP_0(T<\infty)=1$ and $\BP_0(\theta_T B \in \cdot)=\BP^-\odot \nu(\cdot|A) \odot \BP^+$.
\end{theorem}

\begin{proof}[Proof of Proposition \ref{propPitman}]

Unless stated otherwise, we fix $x\in\R$.
For the purpose of this proof, it is convenient to enlarge %for each $x\in\R$ 
the probability space $(\Omega,\mathcal{F},\BP_x)$ to a probability space
$(\Omega',\mathcal{F}',\BP'_x)$, so as to support a Poisson process $\Phi'$ on
$(0,\infty)\times E$ with intensity measure $dt\nu(de)$, independent
of $B$. %We say that an event occurs $\BP'$-a.e., if it occurs
%$\BP'_x$-a.e.\ for every $x\in\R$. 
Define,
\begin{align}\label{tauprime}
\tau'(t):=\int\I\{s\le t\}D(e)\,\Phi'(d(s,e)),\quad t\ge 0.
\end{align}
By \eqref{e11}, $\int \min\{D(e),1\}\,\nu(de)<\infty$. Hence 
\cite[Lemma 12.13]{Kallenberg} (see also
\cite[Proposition 12.1]{LastPenrose17})
shows that the integral \eqref{tauprime}
converges $\BP'_x$-a.e.\ for each $t\ge 0$. Moreover, the process $t\mapsto \tau'(t)$ has
limits from the left, given by
\begin{align*}
\tau'(t-):=\int\I\{s< t\}D(e)\,\Phi'(d(s,e)),\quad t>0.
\end{align*}
Set $\tau'(0-):=0$. Below we will write $\tau'(t)=\tau'(\Phi',t)$ and
$\tau'(t-)=\tau'(\Phi',t-)$.
Equation \eqref{e11} also implies that $\nu(D>r)>0$ for each $r>0$, so that
$\tau'(t)\to\infty$ as $t\to\infty$ holds $\BP'_x$-a.s.

Motivated by \cite[Proposition XII.(2.5)]{RevuzYor99} we now define
a process $W=(W_t)_{t\ge 0}$ as follows. Set $W_0:=0$. Let $t> 0$.
Then there exists $s> 0$ such that $\tau'(s-)\le t<\tau'(s)$.
By definition \eqref{tauprime}, there exists $e\in E$ such
that $\Phi'\{(s,e)\}>0$. Set
$$
W_t:=e(t-\tau'(s-)).
$$
The process $W$ is a measurable function of $\Phi'$. We abuse the
notation and write $W\equiv W(\Phi')$.
By \cite[Proposition XII.(2.5)]{RevuzYor99} and the fact that
\eqref{Poissonexcursion} has the same distribution as $\Phi'$ it follows
that $\BP'_0(W\in\cdot)=\BP_0(B^+\in\cdot)$ is the distribution
of a Brownian motion starting from $0$. In fact, even more is true.
Let $S:=\inf\{t\ge 0:B_t=0\}$ and let $B^S$ be the process
$B$ stopped at $S$; that is $B^S_t:=B_t$ if $t\le S$, and $B^S_t=0$
otherwise. Using the strong Markov property at $S$ together with
\eqref{0inv} and $\ell^0(S)=0$, we obtain that the random measure
%$\Phi$ (defined by \eqref{Poissonexcursion})
$%\begin{align*}
\sum_{s> 0:\tau_{s-}<\tau_s}\delta_{(s-S,\epsilon_{\tau_{s-}})}
$
%\end{align*}
is a Poisson process on $(0,\infty)\times E$ under $\BP'_x$ with intensity measure
$dt\,\nu(de)$, independent of $B^S$.
Define a process $B'=(B'_t)_{t\in\R}$ by
\begin{align*}
B'_t:=
\begin{cases}
B_t,&\text{if $t\le S$,}\\
W_{t-S},&\text{if $t> S$}.
\end{cases}
\end{align*}
Then $B'$ is a measurable function of $(B^S,\Phi')$ and we can write 
$B'\equiv B'(B^S,\Phi')$. Now we have
\begin{align}\label{e12}
\BP'_x(B'\in\cdot)=\BP_x,\quad x\in\R.
\end{align}
Let $(s,e)\in(0,\infty)\times E$. A careful check of the definitions shows that
\begin{align*}
B'_t(B^S,\Phi'+\delta_{(s,e)})=
\begin{cases}
B'_t,&\text{if $t\le S+\tau'(s-)$,}\\
e(t-S-\tau'(s-)),&\text{if $S+\tau'(s-)<t\le S+\tau'(s-)+D(e)$},\\
B'_{t+D(e)},&\text{if $S+\tau'(s-)+D(e)<t$},
\end{cases}
\end{align*}
or
\begin{align}\label{e14}
\theta_{S+\tau'(s-)}B'(B^S,\Phi'+\delta_{(s,e)})
=(\theta_{S+\tau'(s-)}B')^-\odot e\odot (\theta_{S+\tau'(s-)+D(e)}B')^+.
\end{align}

After these preparations, we can turn to the calculation of 
the Palm measure of $N$. 
%Let $x\in\R$ and 
Let $f\colon\Omega\to[0,\infty)$ be measurable. Then
\begin{align*}
\BE_x \sum_{s\in L}f(\theta_sB)\I\{s\in(0,1]\}
&=\BE_x \sum_{s:\tau_{s-}<\tau_s}f(\theta_{\tau_{s-}}B)\I\{\tau_{s-}\in(0,1]\}\\
&=\BE'_x \sum_{s:\tau'(s-)<\tau'(s)}f(\theta_{S+\tau'(s-)}B')\I\{S+\tau'(s-)\in(0,1]\}\\
&=\BE'_x\int f(\theta_{S+\tau'(s-)}B')\I\{S+\tau'(s-)\in(0,1]\}\,\Phi'(ds\times E),
\end{align*}
where we have used \eqref{e12} to get the second identity.
Now we use the independence of $B^S$ and $\Phi'$ along with the Mecke
equation (see e.g.\ \cite[Theorem 4.1]{LastPenrose17}) to obtain that the last expression equals
\begin{align*}
\BE'_x\int_E\int^\infty_0 f(\theta_{S+\tau'(\Phi'+\delta_{(s,e)},s-)}B'(B^S,\Phi'+\delta_{(s,e)})
\I\{S+\tau'(\Phi'+\delta_{(s,e)},s-)\in(0,1]\}\,ds\,\nu(de).
\end{align*}
Clearly, we have $\tau'(\Phi'+\delta_{(s,e)},s-)=\tau'(s-)$. So by \eqref{e14},
the above equals
\begin{align*}
\BE'_x\int_E\int^\infty_0 f((\theta_{S+\tau'(s)}B')^-\odot e\odot (\theta_{S+\tau'(s)}B')^+)
\I\{S+\tau'(s)\in(0,1]\}\,ds\,\nu(de),
\end{align*}
where we have used that $\int\I\{\tau'(s-)\ne \tau'(s)\}\,ds=0$.
Using \eqref{e12} again gives
\begin{align*}
\BE_x \sum_{s\in L}f(\theta_sB)\I\{s\in(0,1]\}
&=\BE_x\int_E\int^\infty_0 f((\theta_{\tau_s}B)^-\odot e\odot (\theta_{\tau_s}B)^+)\I\{\tau_s\in(0,1]\}\,ds\,\nu(de)\\
&=\BE_x\int^\infty_0\int_E f((\theta_tB)^-\odot e\odot (\theta_tB)^+)\I\{t\in(0,1]\}\,\nu(de)\,\ell^0(dt),
\end{align*}
where the second identity comes from a change of variables.
Integrating with respect to $x$ and using \eqref{palmlocaltime}  (for $x=0$) gives
\begin{align*}
\BE \sum_{s\in L}f(\theta_sB)\I\{s\in(0,1]\}
=\BE_0\int f(B^-\odot e\odot B^+)\,\nu(de),
\end{align*}
which is the assertion. 
\end{proof}

\begin{proof}[Proof of Theorem \ref{tpathdeito}]
 Since $\ell^0$ is diffuse and $N_A$ is 
purely discrete, these two random measures are mutually singular.
Moreover, \eqref{palmlocaltime} and Proposition \ref{propPitman}
show that both random measures have intensity $1$.
Proposition \ref{lstt}, \eqref{palmlocaltime}, and Corollary \ref{c12}
imply the assertion.
%\begin{align*}
%\BP_0(\theta_{T'} B \in \cdot)=\BP^-\odot \nu(\cdot|A) \odot \BP^+,
%\end{align*} 
%where $T':=\inf\{t>0: \ell^0[0,t]\le N_A[0,t]\}$. 
%Almost surely, $T'$ cannot be an
%atom of $N_A$. Hence the intermediate value theorem shows that
%$T=T'$, finishing the proof.
\end{proof}

For $t\in\R$, let
\begin{align*}
G_t:=\sup\{s\le t:B_s=0\},
\end{align*}
where $\sup \emptyset:=-\infty$. Below we will write
$G_t\equiv G_t(B)$. Also define $D_t:=D_t(B)=\inf\{s>t:B_s=0\}$.
Note that $\BP(G_t=D_t)=0$ for each $t\in\R$.
In particular,
\begin{align}
\BP=\BE \I\{G_0<D_0\}\I\{B\in\cdot\}.
\end{align}
By \cite[Theorem (v)]{Pitman87}, 
\begin{align*}
\BP=\iint^{D_0}_0\I\{\theta_tB\in \cdot\}\,dt\,d\,\BP_N,
\end{align*}
implying that 
\begin{align}\label{e5.16}
\BP(\theta_{G_0}B\in\cdot)=\int \I\{B\in \cdot\}D_0\,d\,\BP_N;
\end{align}
see assertion (iv) of the above cited theorem.
The measure
$$
\nu':=\BP(\epsilon_{G_0}\in\cdot)=\BP(G_0<D_0,\epsilon_{G_0}\in\cdot)
$$
is known as {\em Bismut's excursion measure}. It follows from
\eqref{e5.16} that 
\begin{align}\label{BismutIto}
\nu'(de)=D(e)\,\nu(de).
\end{align}

Let $A\in\mathcal{E}$ satisfy $0<\nu'(A)<\infty$. Similar to \eqref{NA}, 
define an invariant random measure $N'_A$ by
\begin{align}\label{NA1}
  N'_A:=\nu'(A)^{-1}\int \I\{t\in\cdot\}\,\I\{\epsilon_{G_t}\in A\}\,dt. 
\end{align}
It is easy to see that
\begin{align}\label{NA12}
\BP_{N'_A}(B\in\cdot)
=\nu'(A)^{-1}\BP(B\in\cdot,\epsilon_{G_0}\in A).
\end{align}
We then have the following Bismut counterpart of Theorem \ref{tpathdeito}.

\begin{theorem}\label{tpathBismut}
Let $A \in \mathcal{E}$ be such that $0<\nu'(A)<\infty$
and define the random measure $N'_A$ by \eqref{NA1}. Let
\begin{align*}
T:=\inf\{t>0: \ell^0[0,t]\le N'_A[0,t]\}.
\end{align*}
Then $\BP_0(T<\infty)=1$ and $\BP_0(\theta_{G_T} B \in \cdot)
=\BP^-\odot \nu'(\cdot|A) \odot \BP^+$.
\end{theorem}
\begin{proof} Since $\ell^0$ and Lebesgue measure are mutually singular,
we can apply Theorem \ref{LMT5.1} and Proposition \ref{lstt}.
This gives
$$
\BP_0(\theta_T B\in\cdot)=\BP_{N'_A}(B\in\cdot).
%=\nu'(A)^{-1}\BP(B\in\cdot,\epsilon_{G_0}\in A).%=\BP(\cdot\mid \epsilon_{G_0}\in A).
$$
Since $G_t=G_0(\theta_tB)+t$, $t\in\R$, we hence
have for each measurable $f\colon\Omega\to[0,\infty)$,
\begin{align*}
\BE_0 f(\theta_{G_T} B)
&=\BE_0 f(\theta_{G_0(\theta_TB)} \theta_TB)=\BE_{N'_A} f(\theta_{G_0}B).
\end{align*}
It remains to show that
\begin{align}\label{e6.12}
\BP_{N'_A}(\theta_{G_0} B\in\cdot)=\BP^-\odot \nu'(\cdot|A) \odot \BP^+.
\end{align}
%&=\nu'(A)^{-1}\BE f(\theta_{G_0}B)\I\{\epsilon_0(\theta_{G_0}B)\in A\},
%\end{align*}
%where we have used \eqref{NA12} and the identity
%$\epsilon_{G_0}=\epsilon_0(\theta_{G_0}B)$.
%Using \eqref{e5.16} gives
%\begin{align*}
%\BE_0 f(\theta_{G_T} B)=\nu'(A)^{-1}\BE_{N}  f(B)D_0\I\{\epsilon_0\in A\}.
%\end{align*}
%By Proposition \ref{propPitman},
%\begin{align*}
%\BE_0 f(\theta_{G_T} B)=\nu'(A)^{-1}\iiint f(w_1\odot w_2\odot w_3)D(w_2)\I\{w_2\in A\}
%\,\BP^-(dw_1)\,\nu(dw_2)\,\BP^+(dw_3).
%\end{align*}
%Therefore we obtain the assertion from \eqref{BismutIto}.

%Let $f\colon\Omega\to[0,\infty)$ be measurable. 
Equation \eqref{NA12} implies that 
\begin{align*}
\BE_{N'_A} f(\theta_{G_0} B)
=\nu'(A)^{-1}\BE f(\theta_{G_0}B)\I\{\epsilon_0(\theta_{G_0}B)\in A\}.
\end{align*}
Since $\epsilon_{G_0}=\epsilon_0(\theta_{G_0}B)$, we can use  
\eqref{e5.16} to obtain that
\begin{align*}
\BE_{N'_A} f(\theta_{G_0} B)=\nu'(A)^{-1}\BE_{N}  f(B)D_0\I\{\epsilon_0\in A\}.
\end{align*}
By Proposition \ref{propPitman} this equals
\begin{align*}
\nu'(A)^{-1}\iiint f(w_1\odot w_2\odot w_3)D(w_2)\I\{w_2\in A\}
\,\BP^-(dw_1)\,\nu(dw_2)\,\BP^+(dw_3).
\end{align*}
Hence \eqref{BismutIto} shows that \eqref{e6.12} holds, as required to conclude
the proof.
\end{proof}

\bigskip

\begin{remark}\label{r12}\rm
The identity \eqref{e6.12}, Corollary \ref{c12} and \eqref{BismutIto} show that
the Palm measure of $N'_A$ is (up to a simple shift) a length-biased
version of that of $N_A$.
\end{remark}

\bigskip

\begin{remark}\label{r13}\rm Let $A \in \mathcal{E}$ be such that 
$0<\nu(A)<\infty$ and $0<\nu'(A)<\infty$. For $u\in[0,1]$ let
\begin{align*}
T^u:=\inf\{t>0:uN_A\{0\}+N_A(s,t)\le N'_A[0,t]\}.
\end{align*}
It follows from %Theorem \ref{tmain1} 
Proposition \ref{lstt} that
\begin{align*}
\BE_{N_A}\int^1_0 \I\{\theta_{T^u}B\in\cdot\}\,du=\BP_{N'_A}.
\end{align*}
Since $N_A$ is purely discrete, the proof of this randomized shift-coupling (see Remark \ref{U})
requires Theorem \ref{tmain1}. Theorem \ref{tmain2} would not be enough.
\end{remark}

\bigskip
\noindent
{\bf Acknowledgments:} We would like to thank Jim Pitman for some
illuminating discussions of the topics in this paper.
The first author thanks Steve Evans for supporting his visit
to Berkeley and for giving valuable advice on
some aspects of this work. We also thank the referees for
their helpful comments and advice.

\end{document}